\documentclass{amsart}
\pdfoutput=1
\usepackage[utf8]{inputenc}
\usepackage{amsmath}
\usepackage{amssymb}
\usepackage{amsthm}
\usepackage[T1]{fontenc}
\usepackage{lmodern}

\usepackage{mathrsfs}
\usepackage{enumerate}
\usepackage[labelfont=bf]{caption}
\usepackage{accents}
\usepackage{microtype}
\usepackage{mathtools}

\usepackage{varwidth}

\usepackage{hyperref}
\usepackage{amsmath}
\usepackage{amsthm}
\usepackage{amssymb}
\usepackage{xfrac}
\usepackage{mathtools}
\usepackage{tikz-cd}
\usepackage{enumitem}
\usepackage{comment}
\usepackage{scalerel}
\usepackage{stackengine,wasysym}
\usepackage{csquotes}
\usepackage{todonotes}
\usepackage{cleveref}

\newcommand{\RR}{\mathbb{R}}
\newcommand{\ZZ}{\mathbb{Z}}

\newcommand{\NN}{\mathbb{N}}

\newcommand{\dd}{\mathrm{d}}
\newcommand{\typ}[1]{\left[#1\right]}
\newcommand{\Fatten}[1]{\left<#1\right>}
\newcommand{\mirror}{\star}

\newcommand{\MM}{\mathbb{M}}

\newcommand{\SAT}{\mathcal{D}}

\newcommand{\Graphs}{\mathcal{G}}

\newcommand{\Complexes}{\mathcal{C}}

\newcommand{\Surfaces}{\mathcal{S}}
\newcommand{\Hypers}{\mathcal{H}}

\newcommand{\Riemanns}{\mathcal{R}}

\newcommand{\vol}{\mathrm{vol}}
\newcommand{\Prob}[1]{\mathrm{Prob}\left(#1\right)}

\newcommand{\Ends}[1]{\mathcal{E}\left(#1\right)}
\newcommand{\EndsFin}[1]{\mathcal{E}_{<\infty}\left(#1\right)}
\newcommand{\EndsInf}[1]{\mathcal{E}_{\infty}\left(#1\right)}

\theoremstyle{plain}
\newtheorem{theorem}{Theorem}[section]
\newtheorem{lemma}[theorem]{Lemma}
\newtheorem*{lemma*}{Lemma}
\newtheorem{remark}[theorem]{Remark}
\newtheorem{prop}[theorem]{Proposition}
\newtheorem{cor}[theorem]{Corollary}
\newtheorem{example}[theorem]{Example}
\newtheorem*{theorem*}{Theorem}

\numberwithin{equation}{section}
\theoremstyle{definition}
\newtheorem{definition}[theorem]{Definition}
\newtheorem{construction}[theorem]{Construction}
\newtheorem{notation}[theorem]{Notation}

\title{Ends of stationary metric measure spaces}

\author{Arie Levit and Kfir Silman}

\begin{document}

\begin{abstract}
We prove that stationary random metric measure spaces have $0,1,2$ or a Cantor space of ends. This notion includes stationary random graphs, manifolds and discrete subgroups. In the case of surfaces, we classify all possible homeomorphism types, in analogy with the work of Biringer and Raimbault on unimodular Riemannian  manifolds. Our approach relies on a general \enquote{no geometric core} principle and an analysis of finite versus infinite expected return times.
\end{abstract}

\maketitle

\section{Introduction}

A \emph{metric measure space} is essentially a nice metric space equipped with a nice measure (the precise Definition \ref{def:MMS} is below). Natural examples include locally finite graphs, simplicial complexes, hyperbolic surfaces and manifolds, and more generally Riemannian manifolds equipped with the natural metric and measure.
A \emph{random metric measure space} is a probability measure $\nu$ on the space of \emph{pointed} metric measure spaces. 

In this work we will be considering random walks on random metric measure spaces. To make this notion precise, we introduce \emph{Markov kernels}. These are Borel maps $K$ assigning to every pointed metric measure space $x$ in their domain a transition probability measure $K_x$ on it (see Definition \ref{def:Markov kernel}). A typical sample path of the random walk generated by such a kernel stays on \enquote{one and the same} metric space.

From the topological perspective, we will be interested in studying \emph{spaces of ends}. Roughly speaking, this is an invariant $\Ends{M}$ measuring the different ways a topological space $M$  \enquote{goes off to infinity}.  We  distinguish finite-volume ends $\EndsFin{M}$ and infinite-volume ends $\EndsInf{M}$. The class of topological spaces for which the space of ends can be fruitfully studied are called \emph{Freudenthal}; see \S\ref{sub:the space of ends}.

In its most general form, our main result is the following. It crucially relies on the key assumption of \emph{stationarity} (in the sense of Definition \ref{def:stationary and convolution}).

\begin{theorem}
\label{intro:space of ends general}
Let $\nu$ be a random metric measure space which is stationary with respect to a good Markov kernel. Assume that $\nu$-almost every space is a Freudenthal length space. Then $\nu$-almost surely:
\begin{enumerate}
    \item Either $|\EndsInf{M}| \in \{0,1,2\}$ or $\EndsInf{M}$ is a Cantor space.
    \item If $\EndsFin{M} \neq \emptyset$ then $\overline{\EndsFin{M}} = \Ends{M}$.
\end{enumerate}
\end{theorem}

The notion of a \emph{good} Markov kernel is given in Definition \ref{def:good}. This is a rather mild technical assumption; in the special case of a random walk on a group (or on its Schreier graph), it roughly corresponds to the driving probability measure being symmetric, compactly supported and generating.

The above formulation of Theorem \ref{intro:space of ends general} is perhaps overly abstract. Let us give some concrete examples where it applies.

\begin{cor}
\label{intro:space of ends many examples}
The classification of the space of ends (Theorem \ref{intro:space of ends general}) applies in each of the following particular cases:
\begin{enumerate}
\item $\nu$ is a random Riemannian manifold, stationary with respect to flowing a base point $p$ along a geodesic arc $\gamma_p(v,t)$ in a uniformly random unit tangent direction $v$ and for a uniformly random amount of time $t \in \left[0,1\right]$.
\item Let $G$ be a connected unimodular Lie group. Let $\mu$ be a compactly supported, symmetric, generating and spread-out probability measure on $G$. Let $\nu$ be a random quotient $\Gamma \backslash G$ by a \emph{discrete $\mu$-stationary random subgroup} $\Gamma$; see e.g. \cite{gekhtman2023stationary}. 

\item Let $\Gamma$ be a discrete group and $\mu$ a symmetric, finitely supported and generating measure on  $\Gamma$. 
Let $\nu$ be a random Schreier graph  $\Delta \backslash \mathrm{Cayley}(\Gamma)$ with respect to a \emph{$\mu$-stationary random subgroup $\Delta$}.

\item $\nu$ is a random graph,  stationary with respect to the nearest neighbor random walk (i.e a \emph{stationary random graph}, see e.g. \cite{benjamini2012ergodic}).
\end{enumerate}
\end{cor}

We remark that the Markov kernels in Corollary \ref{intro:space of ends many examples} are good and the spaces are certainly Freudenthal (so that Theorem \ref{intro:space of ends general} is applicable).

\subsection{Stationary random surfaces}
Our work was inspired by the paper of 
Biringer and Raimbault \cite{biringer2017ends}. As was observed in that paper, it turns out that in the two-dimensional case of surfaces much more precise information is available. This depends on the fact that the homeomorphism type of a surface is determined by its genus, number of cusps, its space of ends, and upon knowing the subsets of ends where genus and cusps accumulate.

\begin{theorem}
\label{intro:space of ends surfaces}
Let $\nu$ be a random oriented Riemannian or hyperbolic surface stationary with respect to the random walk described in Corollary \ref{intro:space of ends many examples}.(1) or \ref{intro:space of ends many examples}.(2) respectively. There are $12$ possible homeomorphism types of $\nu$-generic infinite-volume surfaces, as determined by the following table:
\[
\begin{tabular}{c c c c}
Nickname & Genus & Infinite volume ends  & Infinite genus ends\\
\hline
Finite volume & g & 0 & 0 \\
Plane & $0$ & $1$ & $0$ \\
Loch ness monster & $\infty$ & $1$ & $1$\\
The cylinder & $0$ & $2$ & $0$\\
Jacobs ladder & $\infty$ & $2$ & $2$\\
Cantor's tree & $0$ & $\infty$ & $0$\\
Cantor's tree in bloom & $\infty$ & $\infty$ & $\infty$
\end{tabular}
\]
The surfaces may have cusps (i.e. finite-volume ends), in which case the cusps will accumulate to all ends. 
Lastly, in the hyperbolic case the cylinder with no cusps  is excluded, resulting in $11$ possible infinite-volume types.
\end{theorem}

We refer the reader to \cite{ghys1995topologie} and to \cite{biringer2017ends} for nice illustrations of these infinite-type surfaces, and a justification for the imaginative \enquote{nicknames}.

\subsection{Method of proof}
Our results follow from a  \enquote{no geometric core} principle for stationary measures. A \emph{geometric core} is a Borel property $A$ of metric measure spaces, such that every metric measure space satisfying property $A$ admits some subset $U$ with compact boundary and of infinite measure  where property $A$ is \emph{not} satisfied. Roughly speaking, we show that such a geometric core property $A$ is null  with respect to all stationary metric measure spaces.  See Definition \ref{def:core} and Theorem \ref{thm:main markov} below for more details. 

As a special case, we obtain the following \enquote{no bounded core} result: Let $\nu$ be a stationary random pointed metric space which is stationary with respect to a good Markov kernel. If $A$ is a Borel property supported on bounded subsets of infinite measure spaces then $A$ is $\nu$-null; see Corollary \ref{cor:no bounded cores} for the precise statement.

We mention that our \enquote{no geometric core} principle is stronger than the \enquote{no bounded core} principle. To illustrate this, consider the \enquote{comb space} example. It is the space  homeomorphic to  $\RR$ with a copy of $\RR_{\ge 0}$ attached at all integer points. The space of ends of the comb space is homeomorphic to $\ZZ \cup \{\pm \infty\}$ with the obvious topology. Namely, the space of ends is infinite and compact, and all but two ends are isolated. To be able to rule out such a  possibility in Theorem \ref{intro:space of ends general}, we need to be able to regard the unbounded \enquote{handle} of the comb as our geometric core.

Our proof relies on a careful analysis of finite versus infinite expected return times for random walks (see Theorems \ref{thm:MT inequality} and \ref{thm:MT infinity} below). In our level of generality, namely Markov chains on general measure spaces, this information on return times can be found in the textbook of Meyn and Tweedie \cite{meyn2012markov}. For the convenience of the reader, we have decided to include direct proofs of these two theorems in Appendix \ref{appendix}.

Another idea that we make use of in our proof are \enquote{mirrored spaces}. This is a hands on construction, in which a random walk is modified so that it \enquote{reflects} any sample path that reaches a certain \enquote{mirror} subset. For instance, in the above comb space example, the mirrored space construction will allow us to isolate each individual \enquote{tooth} and ignore the \enquote{handle} of the comb. See \S\ref{sec:collapse} for details on this.

\subsection{Related works}

The classical result in the theory of ends of groups is Stalling's theorem \cite{stallings1968torsion}. It implies that the space of ends of (the Cayley graph of) a finitely generated group either has cardinality $0,1,2$ or is the Cantor space.

Stalling's theorem can be extended in various ways. One such result, which served as an inspiration and motivation for ours, is the beautiful paper of Biringer and Raimbault \cite{biringer2017ends}. They study unimodular random Riemannian manifolds, and establish in that context results directly analogous to our Theorems \ref{intro:space of ends general} and \ref{intro:space of ends surfaces}. Interestingly, our classification of the possible homeomorphism types of stationary surfaces coincides with their classification in the unimodular case (i.e. stationarity does not provide any new end space possibilities beyond unimodularity). 

We mention that while Biringer and Raimbault deal specifically with unimodular random Riemannian manifolds (see also  \cite{abert2022unimodular} for an in-depth study of this notion),  their proof is very robust. It essentially applies verbatim in the more general setting of Freudenthal unimodular random metric measure spaces. The end classification is stated explicitly in that generality in \cite[Proposition 5.14]{khezeli2023unimodular}.

The first result (that we are aware of) concerning ends in the stationary case is by Curien. It appears in his enlightening notes on random graphs \cite{curien2017random}. He shows that a random graph stationary with respect to the nearest neighbor random walk has either $0,1,2$ or infinitely many ends \cite[Corollary 24]{curien2017random}. However, it is not immediately clear if the methods in \cite{curien2017random} can be used  to deduce in the infinitely-many ends case that no end is isolated (i.e. that the space of ends is indeed a Cantor space).

Simultaneously to and independently of our work, Yair Hartman and Nadav Kalma have obtained a related result \cite{HK}. They show that the Schreier graph of a stationary random subgroup of a finitely generated group has either $0,1,2$ or infinitely many ends. In addition, they study stationary actions on probability spaces, and show that such actions have  \enquote{no core} (in a suitable dynamical sense). We note that, as in Curien's notes \cite{curien2017random}, the methods of Hartman and Kalma do not show that an infinite space of ends is a Cantor space.


\subsection{Acknowledgments} The authors would like to thank Yehuda Shalom for his interest, support, encouragement and clever advice.
The authors would like to thank Ian Biringer for numerous useful and careful comments and suggestions which greatly enhanced our exposition, and for pointing out a problem in an earlier version.

\section{Random walks on metric measure spaces}

In this section we introduce all the basic notions needed to state our results.

\subsection{Metric measure spaces}

A metric measure space is essentially a space equipped with both a metric and a measure in a compatible manner. More precisely, we use the following definition from \cite[§3]{Bowen15}.
\begin{definition} 
\label{def:MMS}
A \emph{metric measure space} is a triplet $(M, d_M , \vol_M )$ where $(M, d_M )$ is a separable proper metric space and $\vol_M$ is a positive Radon measure on M. The measure $\vol_M$ can be finite or infinite.
A \emph{pointed metric measure space} is a quadruple $(M, d_M , \vol_M , p)$ such that $(M, d_M , \vol_M )$ is a metric measure space and $p \in M$ is a point. 
\end{definition} 

Let $\MM^0$ denote the space of all metric measure spaces. Likewise, let $\MM^1$ denote the space of pointed metric measure spaces. It will be convenient for us to introduce the shorthand notation $\MM = \MM^1$. In addition, whenever convenient we will omit the explicit mention of the metric $d_M$ and the measure $\vol_M$. They are understood as being implicitly associated to the space M.

We endow the space $\MM$ with the topology described in \cite[Definition 5]{Bowen15}.
It is essentially a combination of the pointed Gromov–Hausdorff and the weak-$*$ topologies (the first taking into account the metric structure, and the second the measure).

\begin{notation}
Given a metric measure space $(M,d_M,\vol_M) \in \MM^0$ and a point $p \in M$ it will be convenient to use the notation
$$ \typ{p} = (M,d_M,\vol_M,p) \in \MM.$$
Likewise, for any Borel subset $E \subset M$ we denote
$$ \typ{E} = \{\typ{p} \: : \: p \in E\} \subset \MM.$$
\end{notation}
These notations associate to a point or a subset of a given metric measure space the subset of $\MM$ representing it.

\begin{definition}
\label{def:saturated}    
A \emph{saturated class} is a Borel subset $\SAT \subset \MM$ such that for every pointed metric measure space $(M,p) \in \SAT$ we have $\typ{M} \subset \SAT$.
\end{definition}

In other words, a saturated class is a Borel family of pointed metric spaces which is invariant under a change of basepoint. It can be thought of as being a pullback of some subset under the forgetful map $\MM^1\to\MM^0$.

\begin{example}
\label{example:classes of MMS}
Here are some natural saturated classes:
\begin{enumerate}
    \item Let $\Graphs \subseteq \MM$ consist of pointed connected locally finite graphs equipped with the graph metric. The measure is atomic and supported on the vertices, with the measure of each vertex being equal to its degree. 
    \item Let $\Complexes \subseteq \MM$ consist of pointed locally finite and finite-dimensional simplicial complexes equipped with the intrinsic metric and an atomic measure supported on the vertices, such that the measure of the vertex is its degree in the $1$-skeleton.
    \item Let $\Hypers \subseteq \MM$ consist of pointed hyperbolic manifolds equipped with the hyperbolic distance and the hyperbolic volume measure.
    \item More generally, let $\Riemanns \subseteq \MM$ consist of pointed Riemannian manifolds equipped with the Riemannian distance and the Riemannian volume measure.
\end{enumerate}
These classes satisfy $\Graphs\subseteq \Complexes$ and $\Hypers\subseteq \Riemanns$.
\end{example}

Another basic example of a saturated class is provided by $\typ{M}$ for any single given metric measure space $(M,d_M,\vol_M)\in \MM$.

\begin{remark}
All the spaces we consider in this work will be connected. All the graphs are locally finite. 
\end{remark}



\subsection{Markov kernels on metric measure spaces}
\label{seubsec:Markov kernel}

We develop a notion of random walks on random metric measure spaces. Our proof will crucially rely on studying such random walks and their properties. The random walks we consider are driven by certain Markov kernels, in the following sense. 

\begin{definition}
\label{def:Markov kernel}
Let $\SAT \subset \MM$ be a saturated class. A \emph{Markov kernel} with domain $\SAT$ is a Borel function $$K:\SAT\to \Prob{\SAT }, \quad x \mapsto K_x$$ such that  every pointed metric measure space $x = (M,d_M ,\vol_M ,p) \in \SAT$ satisfies
    $K_x(\typ{M})=1$.
\end{definition}

We understand $K_x$ for some $x = (M,p) \in \MM$ to be the transition probability of the random walk, starting on the space $M$ and at the point $p \in M$.

Given a fixed metric measure space $M \in \MM^0$,  a \emph{Markov kernel on $M$} is understood to be a Markov kernel whose domain is the saturated class $\typ{M}$. This is just a Borel map $K : M \to \mathrm{Prob}(M)$ encoding the transition probabilities of this random walk. A general Markov kernel restricts to a Markov kernel on every space in its domain.

 Assume that $K$ and $N$ are a pair of Markov kernels with a common domain $\SAT$.  Their \emph{convolution product} $N * K$ is the Markov kernel with domain $\SAT$ defined by
$$(N*K)_x(E)=\int_\MM N_y (E) \; \mathrm{d}K_x(y)$$
for every point $x \in \SAT$ and every Borel subset $E \subset \SAT$.  The \emph{convolution power} $K^{*n}$ is defined inductively via $K^{*n} = K*K^{*(n-1)}$ for each $n \in \NN$. 

We proceed to introduce the important notion of stationarity\footnote{In some sources, such as the textbook \cite{meyn2012markov}, the property of being stationary with respect to a Markov kernel is termed \emph{invariance}.} with respect to a Markov kernel. 

\begin{definition}
\label{def:stationary and convolution}
Let $K$ be a Markov kernel with domain $\SAT$. Let $\nu$ be a Borel measure on $\SAT$ (i.e. a measure on $\MM$ satisfying $
\nu(\MM \setminus \SAT)=0$). 
 The convolution $K * \nu$ is the Borel measure on $\SAT$ defined by
$$(K * \nu)(E)=\int_\MM K_x(E) \; \mathrm{d}\nu(x)$$
for each Borel subset $E \subset \SAT$.
The measure $\nu$ is called \emph{stationary with respect to $K$} if $K*\nu =\nu$.
\end{definition}

Note that a measure stationary with respect to a Markov kernel is required to be compatible with it, in the sense that it has to be supported on its domain. For this reason, we will sometimes drop the explicit mention of the domain when talking about stationary measures in this sense.

A measure stationary with respect to some Markov kernel $K$ is also stationary with respect to all of its convolution powers $K^{*n}$.

We point out that Definition \ref{def:stationary and convolution} allows for the measure $\nu$ to be either finite or infinite.


\begin{example}
\label{example:Markov chains}
Here are some examples of Markov kernels.
    \begin{enumerate}
        \item The \emph{nearest neighbor random walk} corresponds to a Markov kernel with domain $\Graphs$. For a pointed graph $x=(G,p) \in \Graphs$ where $p$ is a vertex of $G$, let $K_x$ be the uniform probability measure supported on the neighbors of the vertex $p$ in the graph $G$\footnote{Formally speaking, we are required to define the transition probability $K_x$ for all pairs $(G,p)$ where $p$ is \emph{any} point on the graph $G$. If $p$ lies in the interior of an edge $e =\{u,v\}$, we may take $K_x = \frac{1}{2}(\delta_u + \delta_v)$.}. A probability measure $\nu$ stationary with respect to this kernel is a \emph{stationary random graph}.
        
        \item Let $G$ be a connected unimodular Lie group with a fixed right-invariant metric. Consider the saturated class $\mathrm{GH}(G)$ consisting of all quotients of the form $G/\Gamma$ where $\Gamma \le G$ is a discrete subgroup. Each such quotient is equipped with its natural quotient metric and measure. 
        
        Fix a Borel probability measure $\mu$ on $G$. It defines a Markov kernel $K_
        \mu$ with domain  $\mathrm{GH}(G)$ via the action arising from conjugating the subgroup $\Gamma$. A probability measure $\nu$ stationary with respect to the kernel $K_\mu$ is a \emph{discrete $\mu$-stationary random subgroup of $G$}.
        
	\item Fix a probability measure $\lambda \in \Prob{\RR_{> 0}}$. We define a Markov kernel $K_\lambda$ with domain $\Riemanns$ as follows. For each pointed Riemannian manifold  $x = (M,p) \in \Riemanns$ define $\lambda_p \in \mathrm{Prob}(T_p M) $  by picking a direction uniformly at random and the norm with respect to $\lambda$. The measure $(K_\lambda)_x \in \Prob{M}$ is obtained by pushing forward $\lambda_p$ via the exponential map. 

    \end{enumerate}
\end{example}

\subsection{Good Markov kernels}

We introduce several \enquote{good} properties of Markov kernels that make them easier to work with.

\begin{notation}
Let $K$ be a Markov kernel with domain $\SAT$. Given a metric measure space $x = (M,d_M ,\vol_M ,p) \in \SAT$ and a pair of Borel subsets $A,B\subset M$ we denote
    $$K_{A\to B} = \int_A K_{\typ{q}}(B) \;\mathrm{d} \mathrm{vol}_M(q).$$
\end{notation}

\begin{definition}
\label{def:properties of Markov kernels}
Let $K$ be a Markov kernel with domain $\SAT$.
\begin{itemize}
\item 
\begin{sloppypar}
$K$ is called \emph{intrinsically stationary} if for every pointed metric space \mbox{$(M,p) \in \SAT$} the measure $\vol_M$ is stationary with respect to $K$ regarded as a Markov kernel on $M$, i.e. $K * \vol_M = \vol_M$.
\end{sloppypar}
    \item $K$ is called \emph{intrinsically reversible} if for every pointed metric measure space $(M,p) \in \SAT$ and for every pair of Borel subsets $A,B\subset M$ we have
    $$K_{A\to B} = K_{B\to A}.$$
    
    \item $K$ is called \emph{$R$-bounded} for  some $R>0$ (or, simply \emph{bounded}) if every pointed metric measure space $x=(M,p) \in \SAT$ satisfies $$K_x(\typ{B_{(M,d_M )}(p,R)})=1.$$

     \item $K$ is called \emph{$\varepsilon$-centralized} (or, simply \emph{centralized}) if there is a  constant $\varepsilon>0$ such that for  every pointed metric measure space $x=(M,p) \in \SAT$ there is a constant $c_x > 0$ so that every Borel set $E\subseteq B_{(M,d_M)} (p,\varepsilon )$ satisfies
     $$K_x(\typ{E}) \ge c_x \cdot \vol_M (E).$$
\item $K$ is called \emph{intrinsically irreducible}  if
     every pointed metric measure space $x=(M,p)\in\SAT$ satisfies  $\vol_M \ll \sum_{n=1}^\infty K^{*n}_x$.
\end{itemize}
\end{definition} 

In the classical case of a random walk on a group, the four properties of the kernel being intrinsically reversible, bounded, centralized and intrinsically irreducible roughly correspond to the driving measure being symmetric, compactly supported, absolutely continuous and generating.

It will be convenient for us to introduce the following shorthand terminology.

\begin{definition}
\label{def:good}
A Markov kernel is \emph{good} if it is intrinsically stationary, intrinsically reversible, intrinsically irreducible, bounded and centralized   (all in the sense of Definition \ref{def:properties of Markov kernels}). It is  \emph{$R$-good} if it is good and in particular $R$-bounded for some $R > 0$.
\end{definition}

\begin{example}
\label{example:Markov properties}
We revisit the Markov kernels introduced in Example \ref{example:Markov chains} and analyze their properties.
    \begin{enumerate}
        \item The Markov kernel corresponding to the nearest neighbor random walk with domain $\Graphs$ is intrinsically reversible, intrinsically stationary, intrinsically irreducible and bounded  but not  centralized.  However, the two-step kernel $\frac{K + K^{*2}}{2}$ is good.
         
        \item Let $G$ be a connected unimodular Lie group and $\mu$ a Borel probability measure on $G$. It is intrinsically stationary. The Markov kernel $K_\mu$ with domain $\mathrm{GH}(G)$ is bounded if $\mu$ has compact support. It is centralized if $\mu \ge c \cdot m_G$ for some $c > 0$ on some identity neighborhood, and is  intrinsically reversible if $\mu$ is symmetric. Lastly, it is intrinsically irreducible if, say, it is centralized and also the measure $\mu$ is generating.
        
	\item Fix a probability measure $\lambda \in \Prob{\RR_{> 0}}$. The Markov kernel $K_\lambda$ with domain $\Riemanns$ is intrinsically stationary and intrinsically reversible  \cite[Theorem 3.1]{abert2022unimodular}.
 It is  bounded if $\lambda$ has bounded support. It is centralized and intrinsically irreducible if, say,  $\lambda \ge c \cdot \mathrm{Leb}$ on some interval $\left[0,\varepsilon\right]$. 
    \end{enumerate}
\end{example}

\begin{remark}
In the situation of  Example \ref{example:Markov properties}.(2), assume that $\mu$ satisfies the weaker conditions of being symmetric, compactly supported and spread out. Then the Markov kernel corresponding to some convolution power $\mu^{*n}$ will be bounded, intrinsically reversible and centralized.
\end{remark}







We introduce a  \enquote{measurable saturation} operation, which is shown to behave nicely in the presence of intrinsically irreducible Markov kernels.

\begin{notation}
\label{def:fatten}
Let $A \subset \MM$ be a Borel subset. Denote
$$\Fatten{A} = \{(M, p)\in\MM \: : \: \vol_M (\{q \in M \: : \: \typ{q} \in A \})>0\} \subset \MM.$$
\end{notation}

Note that $\Fatten{A}$ is also a Borel subset of $\MM$.

\begin{prop}\label{prop:O retains 0}
Let $K$ be a intrinsically irreducible Markov kernel with domain $\SAT$. Let $\nu$ be a probability measure stationary with respect to $K$.  Let $A\subseteq\SAT$ be a $\nu$-null set. 
Then $\nu(\Fatten{A})=0$.
\end{prop}

\begin{proof}
The stationarity of the measure $\nu$ implies for all $n \in \NN$ that
$$\int_\SAT \int_A \mathrm{d}K^{*n}_y\;\mathrm{d}\nu(y) = \int_A \; \mathrm{d}(K^{*n}*\nu) = \nu(A)=0.$$
Therefore, to prove the proposition it will suffice to establish the inclusion
$$\Fatten{A} \subseteq \bigcup_{n=1}^\infty \{y\in\SAT \: :\:K^{*n}_y (A)>0\}$$ up to measure zero.
The intrinsic irreducibility assumption says that every point  $y=(M,p)\in \SAT$ satisfies $\vol_M \ll \sum_{n=1}^\infty K_y^{*n}$. If such a point satisfies  $y \in \Fatten{A}$ then we get $\vol_M (\{q \in M \: : \: \typ{q} \in A \})>0$. Hence $K^{*n}_y(A)>0$ for some $n \in \NN$, as required.
\end{proof}


\subsection{Random walks and expected return times}
Throughout this subsection,  fix a Markov kernel (with domain $\SAT$) and a random pointed metric measure space $\nu$ stationary with respect to $K$, i.e. a probability measure $\nu$ on $\MM$ satisfying $K*\nu = \nu$ (and $\nu(\SAT)=1$). Our goal is to introduce random walks driven by $K$.

\begin{notation}
Denote $$\MM^\infty = \{ (M,d_M,\vol_M; \omega = (p_i)_{i\ge0}) \, : \, p_i \in M \}.$$
Namely $\MM^\infty$ is the space of metric measure spaces equipped with a sequence of points $\omega$.
\end{notation}
    
For every point $x = (M;p) \in \SAT$ we  construct a probability measure $\mathbb{P}_K^x$ on the space $\MM^\infty$, corresponding to the distribution of sample paths starting on the space $M$ and at the point $p$. Namely, we set $p_0= p$, and each consecutive point $p_{i+1} \in M$ is drawn randomly from the distribution $K_{p_i}$.

 Let $f$ be a Borel function on the space $\MM^\infty$. For every point $x = (M;p) \in \MM$ we write
    $$\mathbb{E}_{x,K}(f)=\int_{\mathbb{Y}}f \; \mathrm{d}\mathbb{P}_K^x. $$
Namely, this is the expectation of $f$ over a random sample path starting on the space $M$ and at the point $p$. Inside the expectation symbol, we will use $(M,p_i)$ to denote the $i$-th position of the random walk (this is a random variable).

    \begin{notation} 
    Let $A \subset \MM$ be a Borel subset. The \emph{return time to $A$} is a Borel function on $\MM^\infty$ which we denote by $\tau_A$. It is given by
 $$\tau_A (M;\omega=(p_i))=\min\{i\geq 1 \: : \: (M,p_i) \in A\}$$  
 where $\min \emptyset = \infty$ by convention. Note that $\tau_A \ge 1$, namely the circumstance where $p_0 \in A$ does not count as a \enquote{return}.
\end{notation}

The following measure will play an important role in our approach.

\begin{definition}
\label{def:nu_A}
Let $A \subset \MM$ be a $\nu$-measurable subset. We introduce the measure $\nu_A$ on $\MM$ given by
$$\nu_A (E) = \int_A \mathbb{E}_{x,K}\left(\sum_{i=1}^{\tau_A }1_E (\typ{p_i} )\right) \; \mathrm{d} \nu(x)$$
for each Borel subset $E \subset \MM$.
\end{definition}

In other words $\nu_A(E)$ is the expected number of visitations to $E$ of a random sample path starting at a $\nu$-random point in $A$ before it returns to $A$.


\begin{theorem}[{\cite[Theorem~10.4.6]{meyn2012markov}}]
\label{thm:MT inequality}
Let $A \subset \MM$ be a $\nu$-measurable subset with $\nu(A) > 0$. Then $\nu_A \leq\nu$.
\end{theorem}

A complete proof of Theorem \ref{thm:MT inequality} adapted from \cite{meyn2012markov} is given in Appendix \ref{appendix} below. Here, we derive an immediate corollary of Theorem \ref{thm:MT inequality}, saying that almost every point of $A$ has a finite expected return time.

\begin{cor}\label{cor:0 sets}
Let $A \subset \MM$ be a $\nu$-measurable subset.  Then $$\nu(\{x\in A \: : \:\mathbb{E}_{x,K}(\tau_A)=\infty\})=0.$$ 
\end{cor}

\begin{proof}
We assume without loss of generality that $\nu(A) > 0$, for otherwise there is nothing to prove.
Let us apply Theorem \ref{thm:MT inequality} with respect to the set $A$. This gives
    $$\int_A \mathbb{E}_{x,K}(\tau_A)\;\mathrm{d}\nu (x) = \nu_A(\MM) \leq \nu(\MM) = 1.$$
In particular the integral is finite. Hence the integrand $\mathbb{E}_{x,K}(\tau_A)$ must be finite $\nu$-almost surely on $A$.
\end{proof}

Our main results follow from the tension between the above Theorem \ref{thm:MT inequality} in the presence of a \emph{finite} stationary measure, and the following Theorem \ref{thm:MT infinity} for \emph{infinite} stationary measures.

\begin{theorem}[{\cite[Theorem~10.4.10]{meyn2012markov}}]
\label{thm:MT infinity}
Let $\alpha$ be an \emph{infinite} measure on $\MM$ stationary with respect to the Markov kernel $K$. Assume that $K$ is intrinsically irreducible. Let $A \subset \MM$ be a $\alpha$-measurable subset with $\alpha(A)>0$ and such that $\alpha(\MM \setminus \Fatten{A}) = 0$.
Then
$$\nu_A(\MM)=\int_A \mathbb{E}_{x,K}(\tau_A) \;\mathrm{d}\alpha(x)=\infty.$$
\end{theorem}

A complete and detailed proof of Theorem \ref{thm:MT infinity} adapted from the textbook \cite{meyn2012markov} is given in Appendix \ref{appendix} below.

\begin{remark} 
We intentionally use a different notation $\alpha$ for an infinite measure satisfying the particular set of assumptions in Theorem \ref{thm:MT infinity}. We will only have occasion to use this theorem in the situation where $\alpha$ coincides with the measure $\mathrm{vol}_M$ for some specific infinite volume metric measure space and $K$ is a Markov kernel on that space. This is a good example to keep in mind.
\end{remark}

\section{Mirrored metric measure spaces}
\label{sec:collapse}

Let $(M,d_M,\mathrm{vol}_M)$ be a metric measure space and $K$ be a Markov kernel on the space $M$. Assume that $M$ is a length space\footnote{The metric space $M$ is proper (it is a part of the definition of metric measure spaces). Any proper length space is also a geodesic space by the Hopf--Rinow theorem in metric geometry \cite[Proposition I.3.7]{bridson2013metric}} and that that $K$ is intrinsically stationary (i.e. satisfies $K * \vol_M = \vol_M$), intrinsically reversible  and $R$-bounded for some $R > 0$.

Given a particular subset $X \subset M$, we would like to be able to modify random walks on $M$, in such a way that they become \enquote{reflected} when hitting $X$. To do so, we define a \enquote{mirrored} space $\widehat{M}$ equipped with a \enquote{mirrored} kernel $\widehat K$. This construction will depend on the value of the parameter $R$.

Here are the details underpinning the mirroring construction.

\begin{construction}
\label{def:collapse}
 Let $X \subset M$ be a closed subset with compact boundary.  The \emph{$(X,R)$-mirrored metric measure space $(\widehat M,d_{\widehat M},\mathrm{vol}_{\widehat{M}})$} and the \emph{$(X,R)$-mirrored Markov kernel $\widehat K$} on the space $\widehat M$ is defined as follows:
\begin{itemize}
    \item The space $\widehat M$ is the quotient space obtained from $M$ by collapsing the subset $X$ to a point. Denote this special point by $\mirror \in N$.  
    \begin{itemize}
        \item Let $\pi : M \to \widehat M$ denote the quotient map. It sets up a bijection $\pi : M \setminus X \to \widehat M \setminus \{\mirror\}$. Further $\pi(X) = \{\mirror\}$.
    \end{itemize}
    \item The space $\widehat M$ is equipped with the quotient metric $d_{\widehat M}$. The quotient metric induces the topology (by the assumption that the subset $X$ is closed and has a compact boundary).
    \item The space $\widehat M$ is equipped with the measure $\mathrm{vol}_{\widehat M}$ defined as follows:
    \begin{itemize}
            \item The measure $\mathrm{vol}_{\widehat M}$ on $\widehat M \setminus \{\mirror\}$ is induced by pushing forward the measure $\mathrm{vol}_M$ from $M \setminus X$ via the quotient map $\pi$.
        \item Consider the subsets $T = \delta_X^{-1}(\left(0,R\right]) \subset M$ and  $Q = \pi(T) \subset \widehat M$. Note that $Q$ is bounded and $\mirror \in \overline{Q}$.

        \item Set $\mathrm{vol}_{\widehat M}(\{\mirror\}) = K_{M\setminus X \to X}$. Note that $K_{M\setminus X \to X} = K_{T \to X}$ as $K$ is bounded with constant $R$, so that this quantity is finite.
    \end{itemize}
\item The Markov kernel $\widehat K$ on the $(X,R)$-mirrored metric measure space $\widehat M$ is defined as follows:
\begin{itemize}
    \item For any point $y \in \widehat M \setminus\{\mirror\}$ of the form  $y = \pi(z)$ for some $z \in M \setminus X$ we set
    $K_y= \pi _* K_z$.
    \item The transition probability $\widehat K_\mirror$ at the special point $\mirror \in \widehat M$ is defined as follows. In case 
     $K_{X\to M \setminus X } > 0$  we set 
    $$ \widehat K_\mirror(E)=\frac{K_{X\to F}}{K_{X \to M\setminus X}}$$    
    for each Borel set $E  \subset \widehat M \setminus \{\mirror\}$ of the form $E = \pi(F)$ for some Borel set $F \subset M \setminus X$. In addition, we set $\widehat K_\mirror(\{\mirror\})=0$.
     Alternatively, in case  $K_{X \to M \setminus X} = 0$ we simply set  $\widehat K_\mirror = \delta_\mirror$.
    \end{itemize}
    \end{itemize}
\end{construction}

To round up the construction of mirrored metric measure spaces, we verify that the measure $\mathrm{vol}_{\widehat M}$ is stationary with respect to the kernel $\widehat K$ (i.e. the Markov kernel $\widehat K$ is intrinsically stationary on $\widehat M$). In the essential case $K_{X \to M \setminus X} > 0$ we have
$$
(\widehat{K} * \mathrm{vol}_{\widehat{M}})(\{\mirror\}) =  \widehat{K}_{\{\mirror\}\to \{\mirror\}} + \widehat{K}_{{\widehat{M}}\setminus\{\mirror\}\to\{\mirror\}} = 0+K_{M\setminus X\to X} =\vol_{\widehat{M}}(\{\mirror\}).
$$
Further, for each Borel subset $E \subset {\widehat{M}} \setminus \{\mirror\}$ we have
\begin{align*}
(\widehat{K} * \vol_{\widehat{M}})(E) &= \widehat{K}_{\{\mirror\}\to E} + \widehat{K}_{{\widehat{M}}\setminus \{\mirror\} \to E} = \vol_{\widehat{M}}(\{\mirror\}) \cdot \widehat{K}_\mirror(E) + \widehat{K}_{{\widehat{M}}\setminus \{\mirror\} \to E}
\\ &= K_{M\setminus X\to X} \cdot \frac{K_{E\to X}}{K_{M\setminus X\to X}} + K_{M \setminus X \to E } = K_{E \to X } + K_{E \to M \setminus X } \\
&= \vol_{\widehat{M}}(E).
\end{align*} 
In the last equation we have slightly abused the notations and allowed $E$ to denote a Borel subset of ${\widehat{M}} \setminus \{\mirror\}$ as well as its preimage in $M \setminus X$. We have also used in the construction the fact that $K$ is intrinsically reversible to make sure $\vol_{\widehat{M}}(\{\mirror\}) < \infty$. 

In the remaining case 
$K_{X \to M \setminus X} = 0$ the verification of the stationarity of $\vol_{\widehat M}$ with respect to the kernel $\widehat{K}$ is immediate.

\begin{remark}
\label{remark:empty set}
The mirroring construction in Example \ref{def:collapse} can be performed in the special case where $X$ is the empty subset. In that case, the resulting metric measure space $\widehat M$ and Markov kernel $\widehat K$ are identical to the original ones $M$ and $K$.
\end{remark}

\section{No stationary geometric core}
\label{sec:no core}

We define a notion of a geometric core property on the collection of pointed metric measure spaces. Such a property will be encoded by the Borel subset of $\MM$ where it holds. The main result of this section says that geometric core properties must be null with respect to stationary measures.

\begin{definition}
\label{def:core}
A Borel subset $A  \subset \MM$ is  a \emph{geometric core} if for every point $(M,p) \in A$ 
    there is a Borel subset $U \subset M$ with compact boundary such that
    $\typ{U} \subset \MM \setminus A$ and
    $\mathrm{vol}_M(U) = \infty$.
\end{definition}

The formula $\typ{U} \subset \MM \setminus A$ means that the geometric core property $A$ is \emph{not} satisfied on the entire subset $U$.
So, intuitively speaking, whenever a pointed metric space satisfies a geometric core property, it has to admit an infinite-volume part with \enquote{bounded boundary} where the core property fails.

\begin{example}[Bounded cores]
\label{example:bounded core}
Let $A \subset \MM$ be a Borel subset. If every point $(M,d_M,\mathrm{vol}_M;p) \in A$ satisfies $\vol_M(M) = \infty$ and the subset $\{q \in M \: : \: \typ{q} \in A\}$ has bounded diameter with respect to the metric $d_M$ (i.e. it is compact) then $A$ is a geometric core.
\end{example} 

Here is the main result of this work. It applies quite generally, but its statement is a bit technical.

\begin{theorem}[No stationary geometric core]
\label{thm:main markov}
Let $K$ be a good Markov kernel and $\nu$ a random pointed length metric measure space stationary with respect to $K$. Then $\nu(A)=0$ for any geometric core $A \subset \MM$.
\end{theorem}

We recall that a Markov kernel is called good if it is intrinsically stationary, intrinsically reversible, intrinsically irreducible, bounded and centralized. Its domain is some saturated class $\SAT$ of metric measure spaces implicit in the proof.

\begin{proof}[Proof of Theorem \ref{thm:main markov}]
Let $A \subset \MM$ be a geometric core. For each $r > 0$ we define
$$ A_r = \{ (M,q) \in \MM \: : \: \text{there is  $(M,p) \in A$ with $d_M(p,q)< r$} \}$$
so that $A_r \subset \MM$ is a Borel subset (in fact $A_r$ is also a geometric core; we will not use this observation directly). Let $R > 0$ be a constant such that the Markov kernel $K$ is $R$-bounded. 

Let $(M,d_M,\vol_M,p) \in A$ be an element in the geometric core. Let $U \subset M$ be the associated Borel subset as in Definition \ref{def:core}. There is some sufficiently large radius $r = r(M,p)$ so that the open ball $B_r = B_{(M,d_M)}(p,r)$ and the corresponding subsets
$$  V = \{x \in U \: : \: \typ{x} \notin  A_r \},  \;  Q=B_r \cup (U \setminus V)   \quad \text{and} \quad X = M \setminus (Q\cup V  )$$
satisfy that $d_M(V,X) > R$,  $\vol_M(Q) > 0$ and that $X$ is closed. This is possible by the assumption that the boundary of $U$ is compact and since $\vol_M(U) = \infty$. Note that the resulting closed subset $X$ has compact boundary, and $M = X \cup Q \cup V$ is a disjoint union.

Let $(\widehat M,d_{\widehat M},\vol_{\widehat M})$ be the $(X,R)$-mirrored metric measure space obtained from $M$ by collapsing the subset $X$ to a single point $\mirror \in \widehat M$, and $\widehat K$ be the $(X,R)$-mirrored Markov kernel on $\widehat M$; see Construction \ref{def:collapse}. The measure $\vol_{\widehat M}$ is stationary with respect to $\widehat K$. We continue using  $Q$ and $V$ for their respective images in the mirrored space $\widehat M$ (by a slight abuse of notation). Namely $\widehat M = Q \cup V \cup \{\mirror\}$ is a disjoint union.

We  claim that  $\vol_{\widehat M}$-almost every point $q \in Q \subset \widehat M$ satisfies
 \begin{equation} 
 \label{eq:return time inequality}
 \mathbb{E}_{q,\widehat K}(\tau_Q) \le \mathbb{E}_{\typ{q},K}(\tau_{A_r}) + 1.
 \end{equation}
 On the right-hand side  of the inequality, the random walk is performed simply with respect to $K$ regarded as a Markov kernel on its domain in the \enquote{usual} sense.
 On the left-hand side, the random walk is performed with respect to the Markov kernel $\widehat K$ on the individual space $\widehat M$.  The idea is to track these two random walks simultaneously. 
 
 To establish the claim, consider a generic sample path $\omega = (p_n)$ where $p_0 = q \in Q$ and each point $p_n$ lies in the space $\widehat M$. The $R$-boundedness of the kernel $K$ guarantees that, almost surely, if $p_n \in V$ for some $n$ then $p_{n+1} \in V \cup Q$. Observe that $$\typ{Q} \subset A_r \quad \text{and} \quad \typ{V} \subset \MM \setminus A_r.$$
 Hence, there are three mutually distinct possibilities for the sample path $\omega$:
 \begin{enumerate}
     \item $p_n \in V$ for all $n \in \NN$. In this case $\tau_Q(\omega) = \tau_{A_r}(\omega) = \infty$.
     \item $p_n \in V$ for all $ n \in \{1,\ldots,n_0-1\}$ and $p_{n_0} \in Q$ for some $n_0\in\NN $. Then $n_0=\tau_Q(\omega) =  \tau_{A_r}(\omega)$.
     \item $p_1 = \mirror \in \widehat M$. 
     Then $\tau_Q(\omega) = 2$ while $\tau_{A_r}(\omega) \ge 1$.
 \end{enumerate}
 Integrating with respect to the measure on the space of sample paths establishes the desired inequality in the claim. 
 Next, Theorem \ref{thm:MT infinity} applied with respect to the kernel $\widehat K$ on the mirrored metric measure space $\widehat M$, combined with the return-time inequality from the above claim, gives
 $$\int_Q \mathbb{E}_{\typ{q},K}(\tau_{A_r}) \; \mathrm{d} \vol_M (q)\geq
 \int_Q\mathbb{E}_{q,\widehat K}(\tau_Q) \; \mathrm{d} \vol_{\widehat M} (q) - 1 =
 \infty-1 = \infty.$$
 
 To ease our notations set $E_r(x) = \mathbb E_{x,K}(\tau_{A_r})$ througout the rest of the proof. Let $\varepsilon > 0$ be a constant such that the Markov kernel $K$ is $\varepsilon$-centralized. The bounded set $Q$ can be covered by finitely many sets of diameter bounded by $\varepsilon$ (i.e. the intersections of balls with $Q$). Hence there is Borel subset $L \subset Q$ with $\mathrm{diam}_{d_M}(L) \le \varepsilon$ satisfying
$$ \int_{L} E_r(\typ{q}) \; \mathrm{d} \vol_M (q) = \infty.$$
The fact that the Markov kernel $K$ is $\varepsilon$-centralized implies that 
$$ \int_{\MM} E_r(z) \; \mathrm{d}K_{q}(z)=\infty$$ for $\vol_M$-almost all points $q \in L$. Certainly $\vol_M(L) > 0$. 
In particular, it is the case that
$$ \vol_M(\{q \in M \: : \: \typ{q} \in A_r \quad \text{and} \quad \int_{\MM} E_r(z) \; \dd K_q(z) = \infty\}) > 0.$$

Altogether, we have established the following inclusion
$$ A \subset \bigcup_{r \in \NN} A^*_r \quad \text{where} \quad A^*_r =\Fatten{\left\{x \in A_r \: : \: \int_{\MM} E_r(y) \; \dd K_x(y)= \infty\right\}}$$
and $\Fatten{\cdot}$ is the \enquote{measurable saturation} operation (see Notation \ref{def:fatten}). To conclude, note that the subset
 $\{x\in A_r :   E_r(x)= \infty  \}$
 is $\nu$-null for each $r $  by Corollary \ref{cor:0 sets}. Hence also the subset $A_r^*$
 is $\nu$-null for each $r$ by stationarity and by Proposition \ref{prop:O retains 0}. We conclude that $\nu(A) = 0$ as required.
 \end{proof}

The following result is to be compared with \cite[Proposition 23] {curien2017random} in the case of stationary random graphs  and with the main results of \cite{HK} in the case of stationary random subgroups of discrete groups.

\begin{cor}[No stationary bounded core]
\label{cor:no bounded cores}
Let $K$ be a good Markov kernel and $\nu$ be a random pointed length metric measure  space stationary with respect to $K$. Let $A \subset \MM$ be a Borel subset. If every point $(M,d_M,\mathrm{vol}_M;p) \in A$ has $\vol_M(M) = \infty$ and $$\mathrm{diam}_{d_M}(\{q \in M \: : \: (M,q) \in A\}) < \infty$$ then 
$A$ is $\nu$-null.
\end{cor}
\begin{proof}
A Borel subset $A$ with the stated properties is clearly a geometric core (see Example \ref{example:bounded core}). The fact that $\nu(A) = 0$ follows as a special case of Theorem \ref{thm:main markov}.
\end{proof}

\begin{remark}
The  bounded core result (Corollary \ref{cor:no bounded cores}) is a special case of the geometric core result (Theorem \ref{thm:main markov}). However, in the proof of the former we do not need to use mirrored spaces (or rather, the mirroring construction is applied with respect to the empty set $X = \emptyset$, see Remark \ref{remark:empty set}). In particular, the intrinsic reversibility assumption is not needed in Corollary \ref{cor:no bounded cores}.
\end{remark}

\section{Space of ends and geometric cores}
\label{sec:cores}

The notion of a geometric core was introduced in \S\ref{sec:no core} above; see Definition \ref{def:core}. 
In this section we construct several explicit and useful such cores. Most of them will be related to understanding spaces of ends.

\subsection{The space of ends}
\label{sub:the space of ends}
We recall the classical notion of the space of ends. Roughly speaking, it is an invariant measuring all the different ways to go to infinity in a topological space. It was introduced by
Freudenthal in \cite{freudenthal1931enden}. It works well for the following class of spaces.

\begin{definition}
\label{def:Freudenthal}
A Hausdorff topological space $M$ is called \emph{Freudenthal} if $M$ is locally compact, locally connected, connected and separable. 
\end{definition}

Let $M$ be a Freudenthal topological space.   For each compact subset $Q \subset M$ let $\mathcal{U}(Q)$ denote the collection of unbounded (i.e. having non-compact closure) connected components of the complement $M \setminus Q$.  
The  sets $\mathcal U(Q)$ form an inverse system, in the sense that whenever a pair of compact sets satisfies $Q_1 \subset Q_2$ inclusion of connected components determines a map $\mathcal U(Q_2) \to \mathcal U(Q_1)$. 
Note that given a compact subset $Q \subset M$ the collection $ \mathcal U(Q)$ is finite \cite[Lemma 1.1]{raymond1960end}.

\begin{definition}
\label{def:space of ends}
The \emph{space of ends} $\Ends{M}$ of the Freudenthal space $M$ is the inverse limit of the inverse system $\mathcal U(Q)$ where $Q$ ranges over all compact subsets of $M$. An \emph{end neighborhood} of a given end $\zeta \in \Ends{M}$ is an element $V \in \mathcal{U}(Q)$ corresponding to $\zeta$ for some compact subset $Q \subset M$.
\end{definition}

Freudenthal proved that the space of ends $\Ends{M}$ of the Freudenthal space $M$ is a compact, totally disconnected, separable Hausdorff topological space with respect to the topology generated by end neighborhoods \cite{freudenthal1931enden}. In fact, it is the maximal compactification of the space $M$ with those properties\footnote{Further literature on the space of ends is \cite{
raymond1960end,
siebenmann1965obstruction,
peschke1990theory,
hughes1996ends,
guilbault2021endsshapesboundariesmanifold,
axon2025end,
bass2025ends}.}.

Assume further that $M$ is a Freudenthal metric measure space, equipped with the measure $\mathrm{vol}_M$. Let $\EndsInf{M}$ denote the subset of all \emph{infinite-volume ends}, namely ends $\zeta\in \Ends{M}$ all of whose end neighborhoods have infinite $\mathrm{vol}_M$-measure.
Let $\EndsFin{M}$ denote the subset of \emph{finite-volume ends}, namely $\EndsFin{M} = \Ends{M}\setminus \EndsInf{M}$. 
For example, a cusp of a hyperbolic manifold is a finite-volume end, and a funnel is an infinite-volume end. Ends of graphs always have infinite volume. 

Lastly,  consider the following observation, which will be useful below to verify the defining properties of cores. Its proof is immediate from the definitions.

\begin{lemma}
\label{lemma:being properly bounded}
Let $M$ be a Freudenthal space and $Q \subset M$ a compact subset.  Let $V \in \mathcal{U}(B)$ be an end neighborhood. The subset $V$ has compact boundary.
\end{lemma}

\subsection{Four geometric cores}

We define several geometric cores, starting with very general ones, and moving on to more specialized. First, there is a pair of cores dealing with the topology of infinite-volume and of finite-volume ends. Let
$$ A_r^\mathrm{isolated} = \left\{(M;p) \in \MM \: : \: \begin{tabular}{c} $|\mathcal{U}_\infty(B)| \ge 3$ for $B=B_M(p,r)$ and some \\ $\zeta \in \mathcal{U}_\infty(B)$ contains a single infinite-volume end \end{tabular} \right\}.$$
For each $v > 0$ let
$$ A_{r,v}^\mathrm{finite} = \left\{(M;p) \in \MM \: : \: \begin{tabular}{c} the ball $B=B_M(p,r)$ has $\vol_M(B)\ge v$, \\
some $\xi \in \mathcal{U}_{<\infty}(B)$ has $\vol_M(\xi) < v$ and\\
 some $\zeta \in \mathcal{U}_{\infty}(B)$ contains no finite-volume ends
 \end{tabular} \right\}.$$
Here,  as usual, the space $(M;p)$ is implicitly equipped with the metric $d_M$ and the measure $\vol_M$. 

The next core applies only to Riemannian manifolds, and is defined on the corresponding saturated class  $\Riemanns \subset \MM$. For each $R > 0 $ let
$$ A_{r,R}^\mathrm{injrad} = \left\{(M;p) \in \Riemanns \: : \: \begin{tabular}{c} a point in $B = B_M(p,r)$ has injectivity radius  $\le R$ and \\  
some $\zeta \in \mathcal{U}_\infty(B)$ has injectivity radius $\ge R$ at all points
 \end{tabular} \right\}.$$
Here, a Riemannian manifold is understood to be equipped with its Riemannian metric and Riemannian volume measure.

Finally, we define a core which is applicable only to surfaces, for it depends on the notion of genus. Recall that $\Surfaces$ denotes the saturated class in  $\MM$ consisting of Riemann surfaces. Let
$$ A_r^\mathrm{genus} = \left\{(M;p) \in \Surfaces \: : \: \begin{tabular}{c} $B = B_M(p,r)$ has positive genus and  \\
some $\zeta \in \mathcal{U}_\infty(B)$ has zero genus
 \end{tabular} \right\}.$$
Once again, each surface is naturally equipped with its Riemann metric and  area measure. We understand the \emph{genus} of a subset $S$ of a surface to be the supremum of the genus over all compact subsurfaces with boundary contained in $S$.


\begin{prop}
\label{prop:cores}
The properties defined by the Borel subsets $A_r^\mathrm{isolated}$, 
 $ A_r^\mathrm{genus}$, $A_{r,v}^\mathrm{finite}$ for all $v > 0$,  and 
$A_{r,R}^\mathrm{injrad}$ are geometric cores for all values $r,R > 0$.
\end{prop}
\begin{proof}
Fix some $r > 0$. Let $(M,d_M,\vol_M;p)$ be a  metric measure space.  Assume that $M$ belongs to the Borel subset $A_r^\mathrm{isolated}$, $A_{r,v}^\mathrm{finite}$, $A_{r,R}^\mathrm{injrad}$ or $A_r^\mathrm{genus}$ respectively. In all cases, consider the ball
 $B = B_{M,d_M}(p,2r)$. 
Let $U \in \mathcal{U}_\infty(B)$ be the relevant infinite-volume connected 
component of $M \setminus B$. Namely:
\begin{itemize}
    \item $A_r^\mathrm{isolated}$: take  an isolated infinite-volume connected component.
    \item $A_{r,v}^\mathrm{finite}$: take an infinite-volume connected component with no finite-volume ends.
    \item $A_{r,R}^\mathrm{injrad}$: take  an infinite-volume connected component whose injectivity radius is $\ge R$ at all points.
    \item $A_r^\mathrm{genus}$: take an infinite-volume connected component with zero genus.
\end{itemize}
Note that $\left[U\right] \subset \MM \setminus A_r$, namely the property $A_r$ never holds on $U$, where $A_r$ is  one of the four properties $A_r^\mathrm{isolated}$, $A_{r,v}^\mathrm{finite}$, $A_{r,R}^\mathrm{injrad}$ or $A_r^\mathrm{genus}$ respectively. Further $\vol_M(U)=\infty$. The subset  $U$ has compact boundary by Lemma \ref{lemma:being properly bounded}.
This concludes the proof.
\end{proof}


\begin{prop}
\label{prop:on increasing r}
Fix some $r > 0$. Let $(M,d_M,\vol_M;p)$ be a metric measure space satisfying the property
$A_r^\mathrm{isolated}$, 
 $ A_r^\mathrm{genus}$, $A_{r,v}^\mathrm{finite}$  or 
$A_{r,R}^\mathrm{injrad}$
for some $R,v > 0$.
If $q \in M$ is a point with $d_M(p,q) \le \delta$ then $(M;q)$ satisfies the same property with $r$ replaced by $r+\delta$. In particular, if $\vol_M >0 $ then the subset of points in $M$ satisfying the latter property has positive $\vol_M$-measure for all $\delta > 0$ sufficiently large.
\end{prop}
\begin{proof}
Consider a point $ q \in M$ and assume that $d_M(p,q) \le \delta$. The two balls $B_1 = B_M(p,r)$ and $B_2 = B_M(q,r+\delta)$ satisfy $B_1 \subset B_2$ by the triangle inequality. Further, for any component $U_2 \in \mathcal{U}(B_2)$ there is a component $U_1 \in \mathcal{U}(B_1)$ with $U_2 \subset U_1$. The statement follows from these remarks, upon carefully examining the definitions of the cores $A_r^\mathrm{isolated}$, 
$ A_r^\mathrm{genus}$, $A_{r,v}^\mathrm{finite}$, and 
$A_{r,R}^\mathrm{injrad}$.
\end{proof}

\section{Proof of the main results}

We complete the proof of the main classification results stated in the introduction.

\subsection{Stationary random Freudenthal spaces}
Let $K$ be a good Markov kernel (i.e. $K$ is intrinsically stationary, intrinsically reversible, intrinsically irreducible, centralized and $R$-bounded). 
Let $\nu$ be a random metric measure space. Assume that $\nu$ is stationary with respect to $K$ and that $\nu$-almost every space is a Freudenthal length space.

\begin{proof}[Proof of Theorem \ref{intro:space of ends general}]
The first part of Theorem \ref{intro:space of ends general} says that  $\nu$-almost every space $M$ either has $|\EndsInf{M}| \le 2$ or $\EndsInf{M}$ is the Cantor space.
By Brouwer's well-known characterization of the Cantor space (\cite[Theorem 7.4]{kechris2012classical}), this is equivalent to saying that $\nu$-almost surely,  if $|\EndsInf{M}| \ge 3$ then the totally-disconnected space $\EndsInf{M}$ is perfect (i.e. has no isolated points).

Consider the geometric core  $A_r^\mathrm{isolated}$ defined for any $r > 0$; see Proposition \ref{prop:cores}. Topologically speaking, the desired statement is equivalent to showing that
$$ \nu(\{(M;p) \in \MM \: : \: M \cap A_r^\mathrm{isolated} \neq \emptyset \}) = 0$$
for all $r$ sufficiently large.
Proposition \ref{prop:on increasing r}
implies that 
$$ \{(M;p) \in \MM \: : \: M \cap A_r^\mathrm{isolated} \neq \emptyset \} \subset \bigcup_{\delta \ge 0}  \{(M;p) \in \MM \: : \: \vol_M(M \cap A_{r+\delta}^\mathrm{isolated}) > 0 \}. $$
Hence it 
suffices to show that
$ \nu(\Fatten{A_r^\mathrm{isolated}}) = 0$
for all $r$ sufficiently large (for the operation $\Fatten{\cdot}$  see Notation \ref{def:fatten}).
This latter statement follows by combining the \enquote{no geometric core principle} (Theorem \ref{thm:main markov}) and Proposition \ref{prop:O retains 0}.

The second part of Theorem \ref{intro:space of ends general} says that for  $\nu$-almost every space $M$, the condition $|\EndsFin{M}| \neq \emptyset$ implies $\EndsInf{M} \subset \overline{\EndsFin{M}}$. Once again, we need to establish that the geometric core $A_{r,v}^\mathrm{finite}$ satisfies
$$\nu(\{(M,p) \in \MM \: : \: M \cap  A_{r,v}^\mathrm{finite} \neq \emptyset\}) = 0$$ for all $v > 0$ and all $r$ sufficiently large. We may conclude in the same way as in the first part.
\end{proof}

In the process of the above proof we may implicitly assume that $\nu$-almost every space $M$ satisfies $\vol_M(M) = \infty$, for otherwise there is nothing to prove.

\subsection{Examples of stationary random spaces}
\label{subsec:examples}
We now derive Corollary \ref{intro:space of ends many examples}, which deals with various concrete families of spaces, from our more general Theorem \ref{intro:space of ends general}. This amounts to verifying that the involved spaces and kernels enjoy all the properties required in the statement of Theorem \ref{intro:space of ends general}.

\begin{itemize}
\item
Riemannian manifolds are Freudenthal (i.e. connected, locally connected, separable and locally compact)  spaces with the Riemannian metric.
\item
Connected locally-finite graphs are certainly Freudenthal spaces.
\item
Consider the case where $G$ is a connected unimodular Lie group. The group $G$ admits a left-invariant, proper geodesic metric $d_G$ and a left-invariant Haar measure $m_G$. The \emph{Gromov--Hausdorff space} $\mathrm{GH}(G)$ consists of all quotients $\Gamma \backslash G$ where $\Gamma \le G$ is a discrete subgroup. Each such quotient, equipped with its natural quotient metric and measure, is a Freudenthal  space.
\end{itemize}

In the above three cases of stationary random Riemannian manifolds, discrete stationary random subgroups of Lie groups and stationary random graphs, the fact that the Markov kernel in question is good was explained in Example \ref{example:Markov properties} above. Hence Corollary \ref{intro:space of ends many examples} follows directly in cases (1), (2) and (4) there.


It remains to consider the case of Schreier graphs. Let $\Gamma$ be a discrete group with a finite generating set $\Sigma$. Let $G = \mathrm{Cay}(\Gamma,\Sigma)$ be the corresponding Cayley graph. The \emph{Schreier graph} corresponding to a given subgroup $\Lambda \le \Gamma$ is
$$ \mathrm{Sch}(\Gamma,\Lambda;\Sigma) = \Lambda \backslash \mathrm{Cay}(\Gamma,\Sigma).$$

To be able to encode edge-labeled and oriented Schreier graphs as metric measure spaces, we consider the following auxiliary (and somewhat ad-hoc) construction.

First, assume that the generating set $\Sigma$ satisfies $e \notin \Sigma$ and can be written as  $\Sigma = \Sigma_1 \amalg \Sigma_2$ where $\Sigma_1 \cap \Sigma_1^{-1} = \emptyset$ and $\sigma^2=e$ for all $\sigma \in \Sigma_2$. Fix an arbitrary map  $l : \Sigma \to \left(0,\frac{1}{2}\right)$ such that the real numbers $l(\Sigma)$ are pairwise distinct.
\begin{definition}

Given a subgroup  $\Delta \le \Gamma$ we construct the \emph{metric Schreier graph} $G_{\Gamma,\Delta;\Sigma}^l \in \MM$ as follows:
\begin{itemize}
    \item The vertex set of $G_{\Gamma,\Delta;\Sigma}^l$ is $ \Delta \backslash \Gamma$.
    \item The edge set of $G_{\Gamma,\Delta;\Sigma}^l$ is  obtained as follows:
    \begin{itemize}
        \item For each pair $(\Delta\gamma,\sigma)$ with $\Delta\gamma \in \Delta\backslash \Gamma$ and $\sigma \in \Sigma_1$ there is an edge of length $1$ from the  vertex $\Delta \gamma$ to the vertex $\Delta \gamma \sigma$. Attach a stub (i.e. a geodesic segment of length $1$) to this edge at the point of distance $l(\sigma)$ from the first vertex.
        \item  For each pair $(\{\Delta\gamma,\Delta\gamma\sigma\},\sigma)$ with $\Delta \gamma \in \Delta \backslash \Gamma$ and $\sigma \in \Sigma_2$ there is a single edge of length $l(\sigma)$ between the two vertices  $\Delta \gamma$ and $\Delta \gamma \sigma$.
    \end{itemize}
    \item Equip $G_{\Gamma,\Delta;\Sigma}^l$ with the counting measure on the vertices $\Delta \backslash \Gamma$ (the metric Schreier graph is regular, except at the stubs).
\end{itemize}
\end{definition}

The geometry of the metric Schreier graph $G_{\Gamma,\Delta;\Sigma}^l$  completely reconstructs its Schreier graph structure, i.e. the edge labels with orientations. Let $\Graphs^l_{\Gamma,\Sigma}$ denote the resulting saturated class  (essentially, all possible  Schreier graphs $\mathrm{Sch}(\Gamma,\Delta;\Sigma)$ with a geometric encoding). 

\begin{remark}
An alternative and probably more direct and natural method to put Schreier graphs into the framework of metric measure spaces would be to allow for edge markings. In the context of graphs this leads to the notion of random networks studied in \cite{aldous2007processes}. However, this approach would have required us to discuss \enquote{metric networks} (i.e. metric spaces with markings). For the sake of being concise, we have prompted for the ad-hoc approach of encoding the edge labels by means of the edge lengths.
\end{remark}

Let $\mu$ be a finitely supported symmetric generating probability measure on $\Gamma$. This allows us to define a suitable Markov kernel $K_\mu$ with domain $\Graphs^l_{\Gamma,\Sigma}$. Namely, given a vertex $v=\Delta \gamma$ of the metric Schreier graph $G=G^l_{\Gamma,\Delta;\Sigma}\in\Graphs^l_{\Gamma,\Sigma}$ we set $(K_\mu)_{(G,v)}(\Delta\gamma \delta) = \mu(\delta)$ for all $\delta \in \Gamma$. 

For technical reasons, we need to formally define the Markov kernel $K_\mu$ at all points of the graph $G$. This can be done by sending interior points of edges (or of stubs attached to edges) to a random vertex endpoint of that edge. The resulting two-step Markov kernel $\frac{K_\mu + K_\mu^{*2}}{2}$ is good.

We may now deduce the remaining case (3) of Corollary \ref{intro:space of ends many examples} from Theorem \ref{intro:space of ends general}.

\subsection{Stationary random surfaces}

Let us deduce Theorem \ref{intro:space of ends surfaces} dealing with the possible homeomorphism types of stationary random surfaces by means of the \enquote{no stationary geometric core} result (Theorem \ref{thm:main markov}). The key piece of information is the following result.

\begin{theorem}[Ker\'{e}kjart\'{o} \cite{kerekjarto1923vorlesungen}]
\label{thm:kerekj}
Let $M$ be an orientable surface. The homeomorphism type of $M$ is uniquely determined by its genus, its space of ends $\Ends{M}$, the subspace $\EndsInf{M}$ of infinite-volume ends, and by the subspace  of $\EndsInf{M}$ consisting of ends which accumulate genus.
\end{theorem}

\begin{proof}[Proof of Theorem \ref{intro:space of ends surfaces}]
Let $\nu$ be a stationary random surface. We will apply our \enquote{no geometric core} result (Theorem \ref{thm:main markov}) repeatedly with respect to the various geometric cores constructed in \S\ref{sec:cores} in order to analyze the homeomorphism type of a $\nu$-generic surface.

First, by invoking the two geometric cores $A_r^\mathrm{isolated}$ and $A_{r,v}^\mathrm{finite}$ (or rather, at this point, making use of Corollary \ref{intro:space of ends many examples} directly) we see that $\nu$-almost every surface has either $0,1,2$  or a Cantor space of infinite-volume ends, and if there are finite-volume ends then they are dense in the space of all ends. Similarly, by making use of the geometric core $A_r^\mathrm{genus}$ (in a manner analogous to the proof of Theorem \ref{intro:space of ends general}) we see that if the genus of the surface is positive, then all infinite-volume ends have infinite genus. This establishes the full classification as given in the table. We refer the reader to the paper \cite{biringer2017ends} for more information about obtaining this table.

Lastly, if the surfaces are hyperbolic, then the surface homeomorphism type having zero genus and two infinite-volume ends corresponds to the \enquote{hyperbolic cylinder}, namely a quotient of the hyperbolic plane by the action of a single loxodromic element. This case can be ruled out by invoking Theorem \ref{thm:main markov} with respect to the geometric core $A_{r,R}^\mathrm{injrad}$ detecting injectivity radius (once again, analogously to the proof of Theorem  \ref{intro:space of ends general}).
\end{proof}

\appendix

\section{Stationary measures and expected return times}
\label{appendix}

In this work, we make essential use of Theorems \ref{thm:MT inequality} and \ref{thm:MT infinity} on expected return times. Their statements and proofs can be found in the rich and insightful textbook of Meyn and Tweedie \cite{meyn2012markov}.

The textbook \cite{meyn2012markov} covers the general case of Markov kernels on abstract measure space, as is needed in our setting. However, it uses a very specialized language and notations, and it is not easy to extract from it the proof of a single result without reading it carefully and paying special attention to these notations. 

For this reason, and for the convenience of the reader, we present here the complete and \enquote{stand-alone} proofs of Theorems \ref{thm:MT inequality} and \ref{thm:MT infinity}. These are essentially the proofs from \cite{meyn2012markov} adapted to our language.

\begin{proof}[Proof of Theorem \ref{thm:MT inequality}]
Here $K$ is a Markov kernel and $\nu$ is a probability measure stationary with respect to $K$. Further $A \subset \MM$ is a $\nu$-measurable subset with $\nu(A) > 0$. We consider the associated measure $\nu_A$ given in Definition \ref{def:nu_A}.

Let $E \subset \MM$ be any $\nu$-measurable subset. For each $n$ denote
$$ P_A^n (x,E) = 
\mathbb{P}_{K}^x\left(\{\text{$\tau_A \ge n$ and $p_n \in E$}\}\right).$$
We shall prove by induction  that
$$\nu(E)=\int_{A} \sum_{m=1}^n P_A^m (x,E)\;\mathrm{d}\nu(x)+\int_{\MM\backslash A} P_{A}^n(x,E)\;\mathrm{d}\nu(x)$$
holds true for all $n$.
The fact that $K*\nu = \nu$ implies that
$$\nu(E)=\int_{\MM} K_x(E) \; \mathrm{d}\nu(x)$$
which is equivalent to the base case of the induction (i.e. $n = 1$). To get the induction step, we first use the definition of $P^n_A$ and the stationarity of $\nu$ to write
\begin{align*}
\int_{\MM\backslash A} P_{A}^n(x,E)\;\mathrm{d}\nu(x) &= \int_{\MM} \int_{\MM} 1_{\MM\backslash A}(y)\cdot P_{A}^n(y,E)\;\mathrm{d}K_x(y)\;\mathrm{d}\nu(x)\\
&=\int_{\MM} P_{A}^{n+1}(x,E)\;\mathrm{d}\nu(x).
\end{align*}
It follows from the induction hypothesis and from the above computation that 
\begin{align*}
\nu(E)&=\int_{A} \sum_{m=1}^n P_A^m (x,E)\;\mathrm{d}\nu(x)+\int_{\MM\backslash A} P_{A}^n(x,E)\;\mathrm{d}\nu(x)
\\&=\int_{A} \sum_{m=1}^n P_A^m (x,E)\;\mathrm{d}\nu(x)+\int_{\MM} P_{A}^{n+1}(x,E)\;\mathrm{d}\nu(x)
\\&=\int_{A} \sum_{m=1}^{n+1}P_A^m (x,E)\;\mathrm{d}\nu(x)+\int_{\MM\backslash A} P_{A}^{n+1}(x,E)\;\mathrm{d}\nu(x).
\end{align*}
This establishes the induction step for $n+1$. To conclude the proof, we drop the second term and get
$$\nu(E)\geq \sup \int_{A} \sum_{m=1}^{n} P_A^m (x,E) \;\mathrm{d}\nu(x) = \nu_A (E)$$
as required.
\end{proof}

\begin{proof}[Proof of Theorem \ref{thm:MT infinity}]
Let $A$ be an $\alpha$-measurable subset as in the statement of the theorem. Assume by contradiction that
$$\alpha_A(\MM)=\int_A \mathbb{E}_{x,K}(\tau_A)\;\mathrm{d}\alpha(x)<\infty.$$
It follows that 
$$ \alpha(\{x\in A \: :  \: \mathbb{P}_{K}^x(\{\tau_A=\infty\})>0\}) \le \alpha(\{x\in A \: :\:\mathbb{E}_{x,K}(\tau_A)=\infty\})=0.$$
Recall that $\alpha_A \le \alpha$ holds true by Theorem \ref{thm:MT inequality}. The assumption towards contradiction implies that, in fact,   $\alpha_A = \alpha$ when restricted to the set $A$. Indeed, for any $\alpha$-measurable subset $E\subseteq A$ we get
\begin{align*}
\alpha (A)&=\alpha(E)+\alpha(A\backslash E)\geq\alpha_A (E)+\alpha_A (A\backslash E)\\
&=\int_A \mathbb{E}_{x,K}\left(\sum_{j=1}^{\tau_A }1_E (\typ{p_j} )\right)\;\mathrm{d}\alpha(x)+
\int_A \mathbb{E}_{x,K}\left(\sum_{j=1}^{\tau_A }1_{A \backslash E} (\typ{p_j} )\right)\;\mathrm{d}\alpha(x)\\
&=\int_A \mathbb{E}_{x,K}\left(\sum_{j=1}^{\tau_A }1_A (\typ{p_j} )\right)\;\mathrm{d}\alpha(x)
=\int_A \mathbb{P}_{K}^x(\{\tau_A < \infty\}) \; \mathrm{d}\alpha(x)\\
&\geq \alpha(\{x\in A\::\:\mathbb{P}_{K}^x(\{\tau_A=\infty\})=0\})=\alpha(A).
\end{align*}
Hence equality must hold throughout, so that  $\alpha(E)=\alpha_A (E)$. 

Out next task is to prove that $\alpha_A$ is stationary with respect  to $K$.
As in the proof of Theorem \ref{thm:MT inequality}, we denote for each $n$
$$ P_A^n (x,E) = 
\mathbb{P}_{K}^x\left(\{\text{$\tau_A \ge n$ and $p_n \in E$}\}\right).$$
To see that $\nu_A$ is stationary with respect to $K$, consider a measurable subset  $E\subseteq \MM$. Then
    \begin{align*}
    \alpha_A (E)&=\int_A \mathbb{E}_{x,K}\left(\sum_{j=1}^{\tau_A }1_E (\typ{p_j})\right)\;\mathrm{d}\alpha(x)=\int_A \sum_{n=1}^\infty P_A^n (x,E)\;\mathrm{d}\alpha(x)\\
    &= \int_A P_A^1 (x,E) + \sum_{n=2}^\infty P_A^n (x,E)\;\mathrm{d}\alpha(x)\\
    &=\int_A K_x(E)  \;\mathrm{d}\alpha(x) + \int_A \int_{\MM\backslash A}\sum_{n=1}^\infty P_A^n(y,E) \;\mathrm{d}K_x(y)\;\mathrm{d}\alpha(x)\\
    &=\int_A K_x(E)\;\mathrm{d}\alpha_A (x)+\int_{\MM\backslash A}K_x(E)\;\mathrm{d}\alpha_A (x) = (K*\alpha_A)(E).
    \end{align*}
 The equality on the bottom line is by the fact proved above that $\alpha_A =\alpha$ on the subset $A$ for the first integral, and by the fact that $\alpha$ is stationary with respect to $K$ and the definition of $\alpha_A$ for the second integral. We conclude that $\alpha_A$ is stationary with respect to $K$.
    
Consider the difference measure $\alpha_A' =\alpha-\alpha_A$ on the space $\MM$. Note that $\alpha'_A$ is well-defined; while the measure $\alpha$ is infinite, the measure $\alpha_A$ is finite (by the assumption towards contradiction). The measure  $\alpha'_A$ is non-negative (by Theorem \ref{thm:MT inequality}), and $\alpha'_A$ is stationary with respect to $K$ as a difference of two such stationary  measures. Note that 
    $$\alpha'_A (\MM)=\alpha(\MM)-\alpha_A (\MM)=\infty.$$
In particular $\alpha'_A$ is non-zero. 

On the other hand, we have $\alpha'_A (A)=0$ by the above discussion. By the fact that the kernel $K$ is  intrinsically irreducible, we obtain $\alpha'_A(\Fatten{A}) = 0$ from  Proposition \ref{prop:O retains 0}. We arrive at a contradiction to the assumption that $\alpha(\MM \setminus \Fatten{A})=0$.
\end{proof}

\bibliographystyle{alpha}
\bibliography{stationary}

\begin{thebibliography}{Bow15}

\bibitem[AB22]{abert2022unimodular}
Mikl{\'o}s Ab{\'e}rt and Ian Biringer.
\newblock Unimodular measures on the space of all {R}iemannian manifolds.
\newblock {\em Geometry \& Topology}, 26(5):2295--2404, 2022.

\bibitem[AC25]{axon2025end}
Liam Axon and Jack Calcut.
\newblock The end sum of surfaces.
\newblock {\em Contemporary Mathematics}, 812, 2025.

\bibitem[AL07]{aldous2007processes}
David Aldous and Russell Lyons.
\newblock Processes on unimodular random networks.
\newblock {\em Electronic Communications in Probability [electronic only]},
  12:1454--1508, 2007.

\bibitem[BC12]{benjamini2012ergodic}
Itai Benjamini and Nicolas Curien.
\newblock Ergodic theory on stationary random graphs.
\newblock {\em Electron. J. Probab}, 17(93):1--20, 2012.

\bibitem[BC25]{bass2025ends}
William Bass and Jack Calcut.
\newblock Ends and end cohomology.
\newblock {\em Expositiones Mathematicae}, page 125692, 2025.

\bibitem[BH13]{bridson2013metric}
Martin Bridson and Andr{\'e} Haefliger.
\newblock {\em Metric spaces of non-positive curvature}, volume 319.
\newblock Springer Science \& Business Media, 2013.

\bibitem[Bow15]{Bowen15}
Lewis Bowen.
\newblock Cheeger constants and {$L^2$}-betti numbers.
\newblock {\em Duke Mathematical Journal}, 164(3), February 2015.

\bibitem[BR17]{biringer2017ends}
Ian Biringer and Jean Raimbault.
\newblock Ends of unimodular random manifolds.
\newblock {\em Proceedings of the American Mathematical Society},
  145(9):4021--4029, 2017.

\bibitem[Cur17]{curien2017random}
Nicolas Curien.
\newblock Random graphs: the local convergence point of view.
\newblock {\em Unpublished lecture notes. Available at https://www. math.
  u-psud. fr/\~{} curien/cours/cours-RG-V3. pdf}, 2017.

\bibitem[Fre31]{freudenthal1931enden}
Hans Freudenthal.
\newblock {\"U}ber die enden topologischer r{\"a}ume und gruppen.
\newblock {\em Mathematische Zeitschrift}, 33(1):692--713, 1931.

\bibitem[Ghy95]{ghys1995topologie}
{\'E}tienne Ghys.
\newblock Topologie des feuilles g{\'e}n{\'e}riques.
\newblock {\em Annals of Mathematics}, 141(2):387--422, 1995.

\bibitem[GL23]{gekhtman2023stationary}
Ilya Gekhtman and Arie Levit.
\newblock Stationary random subgroups in negative curvature.
\newblock {\em arXiv preprint arXiv:2303.04237}, 2023.

\bibitem[Gui21]{guilbault2021endsshapesboundariesmanifold}
Craig Guilbault.
\newblock Ends, shapes, and boundaries in manifold topology and geometric group
  theory, 2021.

\bibitem[HK26]{HK}
Yair Hartman and Nadav Kalma.
\newblock Cores in stationary actions and ends of stationary random subgroups.
\newblock {\em preprint}, 2026.

\bibitem[HR96]{hughes1996ends}
Bruce Hughes and Andrew Ranicki.
\newblock {\em Ends of complexes}.
\newblock Number 123. Cambridge university press, 1996.

\bibitem[Kec12]{kechris2012classical}
Alexander Kechris.
\newblock {\em Classical descriptive set theory}.
\newblock Springer Science \& Business Media, 2012.

\bibitem[Ker23]{kerekjarto1923vorlesungen}
B~v Ker{\'e}kj{\'a}rt{\'o}.
\newblock Vorlesungen {\"u}ber topologie: I, fl{\"a}chentopologie.
\newblock 1923.

\bibitem[Khe23]{khezeli2023unimodular}
Ali Khezeli.
\newblock Unimodular random measured metric spaces and palm theory on them.
\newblock {\em arXiv preprint arXiv:2304.02863}, 2023.

\bibitem[MT12]{meyn2012markov}
Sean Meyn and Richard Tweedie.
\newblock {\em Markov chains and stochastic stability}.
\newblock Springer Science \& Business Media, 2012.

\bibitem[Pes90]{peschke1990theory}
Georg Peschke.
\newblock The theory of ends.
\newblock {\em Nieuw Archief voor Wiskunde}, 8:1--12, 1990.

\bibitem[Ray60]{raymond1960end}
Frank Raymond.
\newblock The end point compactification of manifolds.
\newblock {\em Pacific Journal of Mathematics}, 10(3):947--963, 1960.

\bibitem[Sie65]{siebenmann1965obstruction}
Laurence Siebenmann.
\newblock {\em The obstruction to finding a boundary for an open manifold of
  dimension greater than five.}
\newblock PhD thesis, Princeton., 1965.

\bibitem[Sta68]{stallings1968torsion}
John Stallings.
\newblock On torsion-free groups with infinitely many ends.
\newblock {\em Annals of Mathematics}, 88(2):312--334, 1968.

\end{thebibliography}

\end{document}